\documentclass[a4paper,11pt]{article}
\usepackage{hyperref}
\usepackage{amsthm,amsmath,amssymb,amsfonts}
\usepackage[numbers, sort & compress]{natbib}

\def\pref#1{(\ref{#1})}

\makeatletter
\def\p@enumii{}
\makeatother

\newtheorem{thm}{Theorem}[section]
\newtheorem{lem}[thm]{Lemma}
\newtheorem{prop}[thm]{Proposition}
\newtheorem{cor}[thm]{Corollary}
\newtheorem{prob}[thm]{Problem}
\theoremstyle{definition}

\newtheorem{ex}[thm]{Example}
\theoremstyle{remark}
\newtheorem{rem}[thm]{Remark}
\numberwithin{equation}{section}

\newcommand{\z}{\mathbb{Z}}
\newcommand{\n}{\mathbb{N}}

\newcommand{\G}{\mathcal{G}}
\newcommand{\X}{\mathcal{X}}
\newcommand{\A}{\mathcal{A}}
\newcommand{\M}{\mathcal{M}}
\newcommand{\N}{\mathcal{N}}
\renewcommand{\S}{\mathcal{S}}

\newcommand{\U}{\mathrm{U}}
\renewcommand{\b}[1]{\overline{#1}}
\newcommand{\rg}{\rangle}
\renewcommand{\lg}{\langle}
\newcommand{\se}{\subseteq}

\newcommand{\sm}{\setminus }

\newcommand{\ifof}{if and only if }

\renewcommand{\l}{\left}
\renewcommand{\r}{\right}
\renewcommand{\b}{\big}
\newcommand{\chr}{\mathrm{char}\,}

\newcommand{\diag}{\mathrm{diag}}

\newenvironment{psmat}
{\left(\begin{smallmatrix}}
	{\end{smallmatrix}\right)}
	
\sloppypar

\begin{document}
\title{Well-covered Unit Graphs of Finite Rings}

\author{Shahin Rahimi$^1$ and Ashkan Nikseresht$^{2}$\\
\it\small Department of Mathematics, College of Sciences, Shiraz University, \\
\it\small 71457-44776, Shiraz, Iran\\
\it\small $^1$E-mail: shahin.rahimi.math@gmail.com\\
\it\small $^2$E-mail: ashkan\_nikseresht@yahoo.com}
\date{}

\maketitle
\begin{abstract}
Let $R$ be a finite ring with identity. The unit graph (unitary Cayley graph) of $R$ is the graph with vertex set
$R$, where two distinct vertices $x$ and $y$ are adjacent exactly whenever $x+y$ is a unit in $R$ ($x-y$ is a unit in
$R$). Here, we study independent sets of unit graphs of matrix rings over finite fields and use them to characterize
all finite rings for which the unit graph is well-covered or Cohen-Macaulay. Moreover, we show that the unit graph of
$R$ is well-covered if and only if the unitary Cayley graph of $R$ is well-covered and the characteristic of $R/J(R)$
is two.
\end{abstract}


\section{Introduction}

Throughout this paper, all rings are assumed to be finite, associative, and with a nonzero identity, and $R$ always denotes a ring. Additionally, all graphs considered here are simple and undirected. We let
$\U(R)$ denote the set of all unit elements in $R$, $J(R)$ the Jacobson radical of $R$, $\mathbb{Z}_n$ the ring of integers modulo $n$, and $I_n$ the identity matrix of size $n$.

Recall that the \emph{unit graph} of $R$, which we denote by $\G'(R)$, is the (undirected) graph with vertex set $R$,
where two distinct vertices $x$ and $y$ are adjacent if and only if $x + y$ is a unit in $R$. For certain rings
(roughly speaking, for $\mathbb{Z}_n$), unit graphs were introduced by Grimaldi in \cite{grimaldi} and later
generalized to arbitrary rings in \cite{unit-ashrafi}. Recently, several researchers have delved into the study of unit
graphs associated with rings; for instance, see
\cite{unit-ashrafi,unit-heydari,unit-maimani,unit-mudaber,unit-pournaki}.

A graph is called \emph{well-covered} if all its maximal independent sets are of the same size. Well-covered
graphs have been extensively studied in graph theory from various perspectives, including algorithmic, algebraic, and structural characterizations (see, for example, \cite{Brown, Finbow, product, WellNew, Our, yassemi, WellSurvey}).

In \cite{unit-pournaki}, well-covered unit graphs of commutative rings are characterized. The main aim of this research
is to extend these results to non-commutative rings by establishing  a connection with unitary Cayley graphs. Recall
that the \emph{unitary Cayley graph} $\G(R)$ of $R$ is the (undirected) Cayley graph of the abelian group $(R,+)$ with
respect to the set $\U(R)$ of unit elements of $R$. In other words, its vertex set is $R$, and two vertices $x$ and $y$
are adjacent if and only if $x-y$ is a unit in $R$. Here, we show that the unit graph of finite ring $R$ is
well-covered if and only if the unitary Cayley graph of $R$ is well-covered and $\chr\b(R/J(R)\b)=2$. Subsequently, by
utilizing the results on well-covered unitary Cayley graphs presented in \cite{Our}, we attain the desired result. A
key step in proving these results is using the Wedderburn-Artin Theorem to relate the structure of $\G'(R)$ for an
arbitrary finite ring $R$ to $\G'(S)$ where $S$ is a product of matrix rings over some finite fields and hence
simplifying the problem.

Given a graph $G$ with vertex set $[n]=\{1,\ldots, n\}$ and edge set $E(G)$, the \emph{edge ideal} of $G$ is the ideal $I(G)=\lg x_ix_j\mid {i,j}\in E(G) \rg$ of $K[x_1,\ldots, x_n]$ where $K$ is a field. A graph
$G$ is called \emph{Cohen-Macaulay} (\emph{CM}, for short) over $K$ when the quotient ring $K[x_1,\ldots, x_n]/I(G)$ is a Cohen-Macaulay ring (see, \cite{CM
ring} or \cite{hibi} for the algebraic background). In the realm of commutative algebra, extensive research has been devoted to the study of edge ideals associated with graphs and CM graphs, see for example \cite{unit-pournaki, tri-free, matchComp, CWalker, GorenCirc, hibi, LineGraphs,
Trung, yassemi} and the references therein. Here, as a corollary to our results on well-covered unit graphs, we present a characterization of rings with a CM unit graph.

\section{The Main Results}
In this section, we aim to characterize finite rings for which the unit graph is well-covered or Cohen-Macaulay.
To that aim, we first recall the concept of semisimple rings.
A \emph{semisimple ring} is an Artinian ring whose Jacobson radical is zero (see \cite[Theorem 15.20]{AndersonFuller}).
Consider a finite ring $R$. Consequently,  $R/J(R)$ is a semisimple ring. The Wedderburn-Artin Theorem states that a semisimple ring  is isomorphic to $M_{n_1}(D_1) \times \cdots \times M_{n_r}(D_r)$ for some division rings $D_i$ and
positive integers $n_i$ (see \cite[Section 13]{AndersonFuller}). Moreover, by Wedderburn's Little Theorem, it is well-known that a finite division ring is a field. Hence, for any finite ring $R$, we have $R/J(R)\simeq M_{n_1}(F_1) \times \cdots \times M_{n_r}(F_r)$ for some finite fields $F_i$ and positive integers $n_i$.

According to \cite[Proposition 3.1]{Our}, for a subset $\A$ of finite ring $R$, $\A$ is a maximal independent set of $\G(R)$ if and only if $\bar{\A}$ is a maximal independent set of $\G\b(R/J(R)\b)$, where $\bar{\A}$ denotes the image of $\A$ in $R/J(R)$, and $\A=\A+J(R)$. This correspondence served as a crucial step in reducing the problem to semisimple rings for unitary Cayley graphs. It is noteworthy that this correspondence does not hold for unit graphs in general. For instance, $\{1,2\}$  forms a maximal independent set of $\G'(\mathbb{Z}_3)$, but its inverse image in $\G'(\mathbb{Z}_9)$, namely $\{1,2,4,5,7,8\}$, is not even an independent set. Nevertheless, this discrepancy does not occur when $2$ is a zero-divisor.

\begin{lem}\label{rtorj}
Suppose $R$ is a finite ring with $2\notin \U(R)$. Then, $\G'(R)$ is a well-covered graph if and only if so is $\G'\b(R/J(R)\b)$.
\end{lem}
\begin{proof}
The proof follows a similar approach to that in \cite[Proposition 3.1]{Our}, with some minor adjustments. Let $\A$ be a subset of $R$. It suffices to show that $\A$ is a maximal independent set of $\G'(R)$ if and only if $\bar{\A}$ is a maximal independent set of $\G'\b(R/J(R)\b)$ and $\A=\A+J(R)$.
First, assume that $\A$ is a maximal independent set of $\G'(R)$. Noting that $u+J(R)\in \U\b(R/J(R)\b)$ iff $u\in \U(R)$, it is straightforward to show that $\bar{\A}$ is a maximal independent set of $\G'\b(R/J(R)\b)$. To establish $\A=\A+J(R)$, consider $a+j \in \A+J(R)$, where $a\in \A$ and $j\in J(R)$. If $a+j$ is adjacent to some $a'\in \A$, then $(a+a')+j$ is a unit of $R$. Consequently, $a+a'$ is a unit of $R$. Since $\A$ is an independent set, this cannot happen unless $a=a'$. Hence, $2a \in \U(R)$ which contradicts the assumption that $2$ is a zero-divisor. Therefore, $\A \cup \{a+j\}$ is an independent set, which, by maximality of $\A$, implies that $a+j \in \A$, and thus $\A=\A+J(R)$.

Now, for the converse, let $a,a'\in \A$. If $a$ is adjacent to $a'$, then $\bar{a}$ is adjacent to $\bar{a'}$, unless $\bar{a}=\bar{a'}$. Since $\bar{\A}$ is an independent set, it follows that $\bar{a}=\bar{a'}$. Hence, $a'=a+j$ for some $j\in J(R)$. However, $a+a'=2a+j\in \U(R)$ implies that $2a\in \U(R)$, leading to a contradiction. Thus, $\A$ is an independent set. For maximality,  assume that $x\in R$ is not adjacent to any element of $\A$. Since $\bar{\A}$ is a maximal independent set, we have $\bar{x}=\bar{a}$ for some $a\in \A$. Hence, $x=a+j$ for some $j\in J(R)$ which implies that $x\in \A+J(R)=\A$, as desired. Thus, the assertion follows.
\end{proof}

\begin{rem}\label{xtoy}
In the proof of Lemma \ref{rtorj}, we indeed showed that if $2\not\in \U(R)$, then $x$ is adjacent to $y$ in $\G'(R)$ if and only if $\bar{x}$ is adjacent to $\bar{y}$ in $\G'\b(R/J(R)\b)$.
\end{rem}

As mentioned, for a finite ring $R$ with $2\in \U(R)$, there is no connection between the maximal independent sets of  $\G'(R)$ and $\G'\b(R/J(R)\b)$ in general. Indeed, we show that $\G'(R)$ is not well-covered whenever $2\in \U(R)$. First, we deal with rings of matrices. Recall that a diagonal matrix whose entries are $+1$ or $-1$ is called a \emph{signature matrix}.

\begin{lem}\label{Sn}
Let $F$ be a finite field with $\chr(F)\ne 2$ and let $n$ be a positive integer. Then,
$\S_n=\{\diag(a_1,\ldots,a_n)\mid a_i=\pm 1\}$ is a maximal independent set of $\G'\b(M_n(F)\b)$.
\end{lem}
\begin{proof}
Assuming $S, S' \in \S_n$ are distinct elements, we can find a row, say the $i$-th row, in which they differ. However, in this case, $S+S'$ contains a zero row, that is, the $i$-th row. Hence, $\S_n$ is an independent set. For maximality, suppose  $A\in M_n(F)\setminus \S_n$ such that $\S_n\cup \{A\}$ is an independent set.
We know that if $S\in \S_n$, then the matrix $S'$  obtained by negating any row of $S$ is also in $\S_n$. Thus, $\det(A+S')=\det(A+S)=0$, since $A\ne S,S'$.
Letting $1\leq m\leq n$, by induction on $m$, we show that for any $S\in \S_n$,  the matrix formed by selecting any $m$ rows from $2I_n$ and the remaining rows from $A + S$ is not invertible. Denote the rows of $A$ and $S$ by $\textbf{a}_i$ and $\textbf{s}_i$, respectively. Assume that $m=1$ and without loss of generality, let the first row of the aforementioned matrix be the first row of $2I_n$. Note that the first $m$ rows of $I_n$ is the matrix $(I_m\mid 0_{m\times(n-m)})$ and $2\textbf{s}_1=\pm (2I_1\mid 0_{1\times (n-1)})$. Using row
linearity of the determinant, it can be said that:
	\begin{align*}
		\pm\left| \hspace*{-0.75em}\begin{array}{rcl}
2I_{1}\hspace*{-2em} & \mid & \hspace*{-2em} 0_{1\times (n-1)}\\
			&\textbf{a}_2+\textbf{s}_2 &\\
		&	\vdots& \\
			&\textbf{a}_n+\textbf{s}_n&
\end{array} \hspace*{-0.25em} \right|
		&=\begin{vmatrix}
			2\textbf{s}_1 \\
			\textbf{a}_2+\textbf{s}_2\\
			\vdots \\
			\textbf{a}_n+\textbf{s}_n
		\end{vmatrix}
		=\begin{vmatrix}
			\textbf{a}_1+\textbf{s}_1\\
			\textbf{a}_2+\textbf{s}_2 \\
			\vdots \\
			\textbf{a}_n+\textbf{s}_n
		\end{vmatrix}-\begin{vmatrix}
			\textbf{a}_1-\textbf{s}_1\\
			\textbf{a}_2+\textbf{s}_2 \\
			\vdots \\
			\textbf{a}_n+\textbf{s}_n
		\end{vmatrix} \\
		&=\det(A+S)-\det(A+S')=0.
	\end{align*} Thus, the hypothesis holds true for $m=1$. Now, for $m>1$, it can be observed that (here in each term, $0$ denotes the zero  matrix with the appropriate size):
$$\pm\left|  \hspace*{-1em}\begin{array}{rcc}
2I_{m} \hspace*{-3em}&\mid& \hspace*{-3.75em} 0\\
			&\textbf{a}_{m+1}+\textbf{s}_{m+1}& \\
			&\vdots& \\
			&\textbf{a}_n+\textbf{s}_n&
\end{array}\hspace*{-1em} \right|=
\left| \hspace*{-.25em}\begin{array}{ccc}
2I_{m-1}\hspace*{-2.25em} &\mid&\hspace*{-2.25em} 0 \\
			&\textbf{2}s_m& \\
			&\vdots &\\
			&\textbf{a}_n+\textbf{s}_n&
\end{array}\hspace*{-.5em} \right|=
\left|\hspace*{-.25em} \begin{array}{ccc}
2I_{m-1}\hspace*{-2.5em} &\mid&\hspace*{-3em} 0 \\
		&	\textbf{a}_m+\textbf{s}_m &\\
		&	\vdots& \\
		&	\textbf{a}_n+\textbf{s}_n&
\end{array}\hspace*{-1em} \right|-
\left|\hspace*{-.25em} \begin{array}{ccc}
2I_{m-1}\hspace*{-2.5em} &\mid &\hspace*{-3em}0 \\
		&	\textbf{a}_m-\textbf{s}_m &\\
		&	\vdots& \\
		&	\textbf{a}_n+\textbf{s}_n&
\end{array} \hspace*{-1em}\right|=0,$$
by induction hypothesis. Therefore, the hypothesis holds for any $m$. As a result, by setting $m=n$, we obtain $\det(2I_n)=0$, leading to a contradiction since $\chr(F)\ne 2$. Thus, $\S_n$ is a maximal independent set of $\G'\b(M_n(F)\b)$.
\end{proof}

Indeed, in Lemma \ref{Sn}, we constructed a maximal independent set of $\G'\b(M_n(F)\b)$, whose size is a power of two. On the other hand, it is not hard to see that the family of matrices in $M_n(F)$ whose first row is zero forms a maximal independent set of size $|F|^{n^2-n}$, which is an odd number. Consequently, we obtain the following result.

\begin{cor}
For any finite field $F$ and positive integer $n$, if $\chr(F)\ne 2$, then $\G'\b(M_n(F)\b)$ is not well-covered.
\end{cor}

Next, we aim to extend the above corollary to semisimple rings.  Initially, we identify certain maximal independent sets of $\G'(R)$.

\begin{lem}\label{mis-unit}
Let $R=R_1\times R_2\times\cdots\times R_t$, where each $R_i$ is a finite ring and each $\M_i$ is a maximal independent set of $\G'(R_i)$. Then the following statements hold true.

\begin{enumerate}
\item \label{mis-unit1} If $\M_i \subseteq R_i\setminus \U(R_i)$, then $R_1\times  \cdots \times R_{i-1} \times \M_i \times R_{i+1} \times \cdots \times R_t$ is a maximal independent set of $\G'(R)$.

\item \label{mis-unit2} If $2\in \U(R)$ and for each $i$ we have $\M_i \subseteq \U(R_i)$, then $\M_1\times \M_2\times \cdots \times \M_t$ is a maximal independent set of $\G'(R)$.
\end{enumerate}
\end{lem}
\begin{proof}
\pref{mis-unit1} For the sake of simplicity, we prove that if $R$ and $S$ are finite rings and $\M\subseteq R\setminus \U(R)$ is a maximal independent set of $\G'(R)$, then $\M\times S$ is a maximal independent set of $ \G'(R\times S)$. To this end, consider $(m,s), (m',s')\in \M\times S$. If $m=m'$, then $(m,s)+(m',s')=(2m,s+s')$ is not a unit of $R\times S$ since $m\in \M \subseteq R\setminus \U(R)$. So, assume that $m\ne m'$. As $\M$ is an independent set, $m+m'$ is not a unit of $R$. Hence, $(m,s)+(m',s')=(m+m',s+s')$ is not a unit of $R\times S$. Therefore, $\M\times S$ is an independent set. For maximality, let $(\M\times S)\cup \{(r,s)\}$ be an independent set, where $r\in R\setminus \M$. It can be observed that for any $m\in \M$, $(r,s)+(m,1-s)=(r+m,1)$ is not a unit of $R\times S$. Consequently, for any $m\in \M$, $r+m$ is not a unit. Since $\M$ is a maximal independent set, it follows that  $r\in \M$, a contradiction from which the result follows.

\pref{mis-unit2} Let $\M=\M_1\times \M_2\times \cdots \times \M_t$. It is not difficult to see that $\M$ forms an independent set of $\G'(R)$. To establish its maximality, suppose $x=(x_1,x_2,\ldots,x_t)\in R\setminus \M$ such that $\M\cup \{x\}$ is an independent set of $\G'(R)$ and $x_1\notin \M_1$. Consequently, there exists $m_1\in \M_1$ such that $x_1+m_1$ is a unit in $R_1$. If all the other $x_i$'s are in $\M_i$, then as $x\in \U(R)$ and $\M_i\se \U(R_i)$, we infer that $x$ is adjacent to $(m_1,x_2,\ldots,x_t)$, which is a contradiction. Therefore, we may assume that $x_2\notin \M_2$. By a similar argument, we obtain that for each $i$ we have $x_i\notin \M_i$, but there is some $m_i$ such that $x_i+m_i$ is unit. Consequently, $x$ is adjacent to $(m_1,\ldots,m_t)\in \M$, resulting in a contradiction. Thus, the assertion follows.
\end{proof}

Noting that $r\in R$ is a unit iff it is not contained in any maximal ideal of $R$, it follows that $r\in \U(R)$ iff $\bar{r}\in \U\b(R/J(R)\b)$. We use this in the following without any further mention.

\begin{lem}\label{rj to r}
Suppose that $R$ is a finite ring and $\bar{\M}$ is a maximal independent set of $\G'\b(R/J(R)\b)$.
\begin{enumerate}
\item \label{rj to r1}
 If $\bar{\M} \subseteq \b(R/J(R)\b)\setminus \U\b(R/J(R)\b)$, then $\M+J(R)$ is a maximal independent set of $\G'(R)$ with size $|\bar{\M}||J(R)|$.

\item \label{rj to r2} If $\bar{\M}\subseteq \U\b(R/J(R)\b)$ and $2\in \U(R)$, then by choosing $\M$ to be a complete set of representatives of $\bar{\M}$, $\M$ forms a maximal independent set of $\G'(R)$ with $|\M|=|\bar{\M}|$.
\end{enumerate}
\end{lem}
\begin{proof}
\pref{rj to r1} Assume that $m_1+j_1$ is adjacent to $m_2+j_2$ in $\G'(R)$ for some $m_1,m_2\in \M$ and $j_1,j_2\in J(R)$. Consequently, $(m_1+j_1)+(m_2+j_2)=(m_1+m_2)+(j_1+j_2)$ is a unit in $R$. By taking images from $R$ to $R/J(R)$, it follows that $\bar{m}_1+\bar{m}_2$ is a unit in $R/J(R)$. This contradicts the independence of  $\bar{\M}$ unless $\bar{m}_1=\bar{m}_2$. Since $\bar{\M}$ consists only of non-unit elements, thus $\bar{m}_1+\bar{m}_2=2\bar{m}_1$ is not a unit element of $R/J(R)$, a contradiction. Therefore, $\M+J(R)$ is an independent set of $\G'(R)$. To prove that it is maximal, let $r \in R$. If $\bar{r} \in \bar{\M}$, then $r-m\in J(R)$ for some $m\in \M$. Hence $r\in \M+J(R)$. So, assume that $\bar{r}\notin \bar{\M}$. Since $\bar{\M}$ is maximal, there exists $m\in \M$ such that $\bar{r}$ is adjacent to $\bar{m}$. Thus, $\bar{r}+\bar{m}$ is a unit in $R/J(R)$. Consequently, $r+m$ is a unit in $R$, implying that $r$ is adjacent to $m\in \M+J(R)$, as desired.

\pref{rj to r2} For $\bar{\M}=\{\bar{m}_1,\ldots,\bar{m}_t\}$, we choose $\M$ to be $\{m_1,\ldots,m_t\}$ such that  $\bar{m}_i\ne\bar{m}_j$ for each $i\ne j$. Under this assumption, if $m_i$ is adjacent to $m_j$ for $i\ne j$, then $\bar{m}_i$ is adjacent to $\bar{m}_j$, which is absurd. Hence, $\M$ is an independent set of $\G'(R)$. To prove its maximality, let $\M\cup \{r\}$ be an independent set, where $r\in R\setminus \M$. As a result, for all $m\in M$, we have $m+r\not\in \U(R)$, implying $\bar{m}+\bar{r}\not\in \U\b(R/J(R)\b)$. Since $\bar{\M}$ is a maximal independent set, we infer that $\bar{r}=\bar{m}$ for some $m\in \M$. Consequently, $2\bar{m}$ is not a unit element of $R/J(R)$, leading to a contradiction, since $\bar{\M}\subseteq \U\b(R/J(R)\b)$ and $2\in \U(R)$. Hence, the result follows.
\end{proof}

Now, we can show that if the unit graph of $R$ is well-covered, then $2$ must be a zero-divisor in $R$.

\begin{prop}\label{2unit not}
If $R$ is a finite ring and $2\in \U(R)$, then $\G'(R)$ is not well-covered.
\end{prop}
\begin{proof}
Given that $R$ is a finite ring, the quotient ring $R/J(R)$ is semisimple and isomorphic to $M_{n_1}(F_1)\times \cdots \times M_{n_r}(F_r)$ for some finite fields $F_i$ and positive integers $n_i$. For each $i$, let $\S_i$ be the set of signature matrices in $M_{n_i}(F_i)$. By Lemma \ref{Sn}, $\S_i$ is a maximal independent set of $\G'\b(M_{n_i}(F_i)\b)$ with $|\S_i|=2^{n_i}$. By virtue of Lemma \ref{mis-unit}, $\S=\S_1\times\cdots\times\S_r$ forms a maximal independent set of $\G'\b(R/J(R)\b)$ contained in $\U\b(R/J(R)\b)$ with size $2^{n_1}\times\cdots\times2^{n_r}$. On the other hand, the family $\M_1$ of matrices in $M_{n_1}(F_1)$, whose first row is zero, is a maximal independent set of $\G'\b(M_{n_1}(F_1)\b)$. Consequently, by Lemma \ref{mis-unit}, $\M=\M_1\times M_{n_2}(F_2)\times \cdots \times M_{n_r}(F_r)$ is a maximal independent set of $\G'\b(R/J(R)\b)$ contained in $\b(R/J(R)\b)\setminus \U\b(R/J(R)\b)$. Since $2\in \U(R)$, it follows that $2\in \U\b(R/J(R)\b)$. Hence, $2\nmid |\M_1|$ and for all $i$ we have $2\nmid |M_{n_i}(F_i)|$. Consequently, $2\nmid |\M|$, which implies that $|\M||J(R)|\ne |\S|$. Now,  by applying Lemma \ref{rj to r},
we establish the existence of two maximal independent sets of $\G'(R)$ with different sizes. This implies that $\G'(R)$ is not well-covered.
\end{proof}

Next, we consider the well-coveredness of $\G'(R)$ for rings $R$ in which $2$ is a zero-divisor. First, we need a lemma.

\begin{lem}\label{r times s}
If $R$ and $S$ are finite semisimple rings with $\chr(R)=2\ne \chr(S)$, then $\G'(R\times S)$ is not well-covered.
\end{lem}
\begin{proof}
Assume that $R\simeq M_{n_1}(F_1)\times\cdots\times M_{n_r}(F_r)$ and $S\simeq M_{m_1}(E_1)\times \cdots \times M_{m_s}(E_s)$ for some finite fields $F_i$ and $E_i$ with $\chr(F_i)=2\ne \chr(E_i)$ and positive integers $n_i$ and $m_i$. Let $\M_1$ be the maximal independent set of $\G'\b(M_{n_1}(F_1)\b)$ as in Proposition \ref{2unit not}. With the aid of Lemma \ref{mis-unit}, we see that $\M=\M_1\times M_{n_2}(F_2)\times \cdots \times M_{n_r}(F_r)$ is a maximal independent set of $\G'(R)$. Consequently, $\M\times S$ is a maximal independent set of $\G'(R\times S)$ with size $|\M||S|$.

Now, we construct another maximal independent set with a different size. In order to do this, let $\X$ be the set of all non-unit elements of $S$, and set $\N=(R\times \{0\})\cup (\M\times \X)$. Since unit elements of
$R\times S$ are exactly of the form $(a,b)$, such that $a\in \U(R)$ and $b\in \U(S)$, it is not hard to see that
$\N$ is an independent set. If $\N\cup \{(x,y)\}$ is an independent set for some $(x,y)\in \b(R\times S\b)\sm \N$,
then $(x,y)+(1-x,0)=(1,y)$ implies that $y$ is not a unit element of $S$ and hence $x\notin \M$. Now, we prove the following claim.

\textbf{Claim:}
If $y$ is a nonzero non-unit element of the finite semisimple ring $S$, then there exists some non-unit $z\in S$ such that $y+z$ is a unit in $S$.

\textbf{Proof of the claim:}
Since $S\simeq M_{m_1}(E_1)\times \cdots \times M_{m_s}(E_s)$, let $y=(Y_1,Y_2,\ldots,Y_s)$, where $Y_i \in M_{m_i}(E_i)$. If $Y_i=0$, set $Z_i=I_{m_i}$ and if $Y_i$ is invertible, set $Z_i=0$. Else, $Y_i$ is a nonzero non-unit element of $M_{m_i}(E_i)$. Using Gaussian elimination method, we can find invertible matrices $U$ and $V$ such that $Y_i=U	\begin{psmat}
	I_t & 0 \\
	0 & 0
\end{psmat}V$ where $1 \leq t <m_i$.
Thus $ 1 \leq m_i-t< m_i$.
Set $Z_i=U\begin{psmat}
	0 & 0 \\
	0 & I_{m_i-t}
\end{psmat} V$. Clearly, $Z_i$ is not invertible. Moreover, $Y_i+Z_i=UI_{m_i}V=UV$, which is invertible.
Therefore, in all three cases, $Y_i+Z_i$ is invertible. Now, since $y$ is not zero, we can find some nonzero $Y_i$, for which $Z_i$ is not invertible. Hence, $z=(Z_1,Z_2,\ldots,Z_s)$ is a non-unit element of $S$ such that $y+z$ is a unit.

Now, $\M$ being a maximal independent set of $\G'(R)$ and $x\notin \M$ implies that there exists some $m\in \M$ such that $x+m\in \U(R)$. But then $(m,z)\in \N$ and
$(x,y)+(m,z)=(x+m,y+z)$ is a unit element of $R\times S$, that is, $(x,y)$ is adjacent to $(m,z)\in \M\times \X \subseteq \N$,
contradicting the independence  of $\N\cup\{(x,y)\}$. Therefore, $\N$ is a maximal independent set.

If $\G'(R\times S\b)$ is
well-covered, then the maximal independent sets $\M\times S$ and $\N$ must have the same size, implying $|S|=|\N|/|\M|$. Since $(R\times \{0\}) \cap (\M\times \X)=\M\times \{0\}$, it can be seen that: $$|S|=|\N|/|\M|=\b(|R|+|\M||\X|-|\M|\b)/|\M|=|R|/|\M|+|\X|-1.$$
From $|S|=|\X|+|\U(S)|$, we obtain $|\U(S)|=|R|/|\M|-1$. Now, we show that the two sides of the last equality have different parity.
Since all the $F_i$'s are of characteristic $2$, $|R|$ is a power of two, making $|R|/|\M|-1$ an odd integer. On the other hand, for each $i$, we have $\chr(E_i)\ne 2$, hence
$$|\U\b(M_{m_i}(E_i)\b)|=\b(|E_i|^{m_i}-1\b)\b(|E_i|^{m_i}-|E_i|\b)\cdots \b(|E_i|^{m_i}-|E_i|^{m_i-1}\b)$$
is an even integer. Consequently, $|\U(S)|=\b|\U\b(M_{m_1}(E_1)\b)\b|\cdots \b|\U\b(M_{m_s}(E_s)\b)\b|$ is an even integer, which is a contradiction. Thus, the result follows.
\end{proof}

Next, we present a necessary and sufficient condition under which unit graphs and unitary Cayley graphs coincide.

\begin{prop}\label{g=g'}
If $R$ is a finite ring, then $\G(R)=\G'(R)$ if and only if $\chr\b(R/J(R)\b)=2$.
\end{prop}
\begin{proof}
Using Remark \ref{xtoy} and \cite[Proposition 3.1]{Our}, the proof of the converse is routine. For the implication,
suppose $\G(R)=\G'(R)$. Thus, for distinct elements $x,y\in R$ we have $x-y\in \U(R)$ if and only if $x+y\in \U(R)$.
Consequently, for any positive integer $k$ we have $2k-1\in \U(R)$ if and only if $2k+1 \in \U(R)$. As $1$ is a unit,
it follows that every odd integer is a unit element of $R$, implying that $\chr\b(R/J(R)\b)=2$.
\end{proof}

Finally, we determine when the unit graph $\G'(R)$ is well-covered.

\begin{thm}\label{wellunit}
Let $R$ be a finite ring. Then, the unit graph of $R$ is well-covered if and only if the unitary Cayley graph of $R$ is  well-covered and $\chr\b(R/J(R)\b)=2$.
\end{thm}
\begin{proof}
If the unit graph of $R$ is well-covered, Proposition \ref{2unit not} implies that $2$ is a zero-divisor. Considering
Lemma \ref{rtorj}, we conclude that $R/J(R)$ is also well-covered. Given that $R/J(R)$ is semisimple, we can find
finite fields $F_i$ and positive integers $n_i$ such that $R/J(R)\simeq M_{n_1}(F_1)\times\cdots\times M_{n_r}(F_r)$.
Since 2 is a zero-divisor, for at least one $j$ we must have $\chr F_j=2$. Now, if for some $i$ the characteristic of
$F_i$ is not equal to $2$, a contradiction arises by taking Lemma \ref{r times s} into account. Therefore, for each
$i$, we have $\chr(F_i)=2$. Consequently, $\chr\b(R/J(R)\b)=2$. Moreover, according to Proposition \ref{g=g'}, the unit
graph and the unitary Cayley graph coincide. Hence, the only if implication follows. Proposition \ref{g=g'} establishes
the reverse implication, completing the proof of the theorem.
\end{proof}

\begin{cor}\label{wellunit2}
	Let $R$ be a finite ring. Then, the unit graph $\G'(R)$ is well-covered if and only if either
	\begin{enumerate}
		\item $R/J(R)$ is isomorphic to one of $F$, $F\times F$, or $M_2(F)$ for some finite field $F$ with characteristic $2$;
		\item or $R/J(R)$ is isomorphic to $\mathbb{Z}_2^k$ for some $k\in \mathbb{N}$.
	\end{enumerate}
\end{cor}
\begin{proof}
The result follows from Theorem \ref{wellunit} and \cite[Theorem 3.5]{Our}.
\end{proof}

In addition to examples provided in \cite[Example 3.6]{Our}, we present several non-commutative rings (satisfying each
condition stated in the above corollary) for which the unit graph is well-covered. The idea is derived from
\cite[Example 2.10]{ours}, which utilizes the concept of group algebras. Recall that by letting $F$ be a finite field
and $G$ a finite group, a \emph{group algebra} $F[G]$, is defined as the ring of all (formal) linear combinations of
elements of $G$ with coefficients in $F$ with addition and multiplication defined as
\begin{align*}
\sum_{g\in G} a_gg + \sum_{g\in G} b_gg & = \sum_{g\in G} (a_g+b_g)g \\
\l(\sum_{g\in G} a_gg \r)\cdot \l(\sum_{g\in G} b_gg\r) & = \sum_{g\in G} \l(\sum_{x\in G} a_xb_{x^{-1}g}\r)g,
\end{align*}
where $a_g,b_g\in F$ for each $g\in G$.
\begin{ex}
Let $F$ be a finite field with characteristic $2$ and $G$ a finite non-abelian $2$-group. Consider the group algebra
$R=F[G]$, which is a non-commutative ring. According to  \cite[Proposition 19.11]{lam}, we see that $R$ is a local ring
and $R/J(R)\simeq F$. Moreover, by considering $R'=R\times R$, we observe that  $R'/J(R')\simeq  R/J(R)\times R/J(R)
\simeq F\times F$. Furthermore, if we assume that $S=\b(\mathbb{Z}_2[G]\b)^k$ for some $k\in \mathbb{N}$, then
$S/J(S)\simeq \mathbb{Z}_2^k$. Also, by letting $T=M_2(R)$, we infer that $T/J(T)\simeq M_2\b(R/J(R)\b)\simeq M_2(F)$.
Therefore, for all the rings $R$, $R'$, $S$ and $T$, the unit graph is well-covered. The key step in this example is to
find a finite non-commutative local ring with characteristic two, so any other ring with this property would provide a
new family of examples.
\end{ex}

Considering a graph $G$, it is evident that the family of all independent sets $\Delta(G)$ of $G$ forms a simplicial
complex, that is called the \emph{independence complex} of $G$ (for the definition of simplicial complex and its
properties, we refer the reader to \cite{hibi}). A graph $G$ is called \emph{CM} (resp. \emph{Gorenstein},
\emph{shellable}) when the independence complex $\Delta(G)$ is CM (resp. Gorenstein, shellable).

\begin{cor}\label{UnitCM}
Suppose that $R$ is a finite ring. Then, the unit graph of $R$ is CM \ifof it is shellable \ifof either $R$ is a field
with characteristic $2$ or $R\simeq \z_2^k$ for some $k\in \n$. Also, this graph is Gorenstein \ifof $R\simeq \z_2^k$ for some $k\in \n$.
\end{cor}
\begin{proof}
It is well-known that a CM graph is well-covered (for instance, see \cite[Lemma 9.1.10]{hibi}).
So, this is an immediate consequence  of Theorem \ref{wellunit} and \cite[Theorem 3.9 \& Corollary 3.10]{Our}.
\end{proof}

In the end, we present a problem related to our work for further research. Notions of the unit graph and the unitary
Cayley graph of a ring have been extended and generalized in various ways. One of which is as follows and is introduced for
commutative rings in \cite{naghi}. Let $R$ be a ring. We define the \emph{generalized unit graph} of $R$ to be the graph with
vertex set $R$ in which two distinct vertices $x$ and $y$ are adjacent whenever there exists a unit element $u$ of $R$
such that $x + uy$ is a unit of $R$. It is not hard to see that for a finite field $F$, the generalized unit graph of
$F$ is a complete graph and therefore well-covered.

\begin{prob}
Characterize all finite rings for which the generalized unit graph is well-covered or CM.
\end{prob}


\end{document}